\documentclass[pdflatex,sn-mathphys-num]{sn-jnl}

\usepackage{graphicx}%
\usepackage{multirow}%
\usepackage{amsmath,amssymb,amsfonts}%
\usepackage{amsthm}%
\usepackage{mathrsfs}%
\usepackage[title]{appendix}%
\usepackage{xcolor}%
\usepackage{textcomp}%
\usepackage{manyfoot}%
\usepackage{booktabs}%
\usepackage{algorithm}%
\usepackage{algorithmicx}%
\usepackage{algpseudocode}%
\usepackage{listings}%
\usepackage{enumitem}

\DeclareMathOperator{\Real}{Re}

\DeclareMathOperator{\gcdOp}{gcd}

\DeclareMathOperator{\rank}{rank}
\DeclareMathOperator{\Ker}{Ker}
\DeclareMathOperator{\Hess}{Hess}

\DeclareMathOperator{\Aut}{Aut}
\DeclareMathOperator{\Stab}{Stab}
\DeclareMathOperator{\Vol}{Vol}
\newcommand{\C}{\mathbb{C}}
\newcommand{\R}{\mathbb{R}}
\newcommand{\HH}{\mathbb{H}}
\newcommand{\OO}{\mathbb{O}}
\newcommand{\A}{\mathbb{A}}

\setlist{nosep}

\theoremstyle{thmstyleone}%
\newtheorem{theorem}{Theorem}[section]
\newtheorem{lemma}[theorem]{Lemma}
\newtheorem{corollary}[theorem]{Corollary}
\newtheorem{proposition}[theorem]{Proposition}
\newtheorem{maintheorem}{Main Theorem}

\theoremstyle{thmstylethree}%
\newtheorem{definition}[theorem]{Definition}
\newtheorem{remark}[theorem]{Remark}

\begin{document}

\title[Dynamical Non-Commutative Algebraic Geometry]{Dynamical Non-Commutative Algebraic Geometry:\\ Inflation, Bifurcation, and the Dynamics of Collapse across Division Algebras}

\author*[1,2]{\fnm{Pau} \sur{Amaro Seoane}}\email{amaro@upv.es}

\affil*[1]{\orgdiv{Department de Matemàtiques Aplicades}, \orgname{Universitat Politècnica de València}, \orgaddress{\street{C/Vera s/n}, \city{València}, \postcode{46022}, \country{Spain}}}

\affil[2]{\orgdiv{Max Planck Institute for Extraterrestrial Physics}, \orgaddress{\street{Giessenbachstra{\ss}e 1}, \city{Garching}, \postcode{85748}, \country{Germany}}}

\abstract{We develop a framework for dynamical non-commutative algebraic geometry (DNCAG) by analyzing the evolution and stability of polynomial root manifolds in real normed division algebras ($\HH$ and $\OO$). We establish a Generalized Inflation Theorem, demonstrating that for central polynomials, the root set forms a homogeneous space $G/H$, where $G$ is the automorphism group of the algebra ($SO(3)$ for $\HH$, $G_2$ for $\OO$). This mechanism generates continuous geometry from non-commutativity. We analyze the dynamics under central modulation (breathing modes), classifying topological bifurcations ($\Delta=0$). We then analyze the topological collapse induced by non-central perturbations, governed by symmetry reduction. We utilize the Localization Theorem (Gordon-Motzkin) to explain the alignment of roots with coefficient subalgebras. We formalize the dynamics of collapse using gradient flow on the potential landscape $\mathcal{V}(x) = \|P(x)\|^2$, characterizing it as a deformation retract and proving that the collapse timescale exhibits critical slowing down with quadratic scaling ($T_{\rm collapse} \propto \epsilon^{-2}$). Finally, we introduce a thermodynamic formalism, proving an Entropy Scaling Law that rigorously characterizes the collapse as a symmetry-breaking phase transition.}

\keywords{Non-Commutative Geometry, Division Algebras, Dynamical Systems, Bifurcation Theory, Gradient Flow, Symmetry Breaking, Quaternions, Octonions}

\maketitle

\section{Introduction}

The study of polynomial roots over non-commutative division algebras has a rich history, dating back to the foundational work on quaternions ($\HH$) by Niven \cite{niven1941equations} and Eilenberg and Niven \cite{eilenberg1944fundamental}. A key observation, formalized by Gordon and Motzkin \cite{gordon1965zeros}, is that the algebraic structure of the underlying space profoundly affects the geometry of the solution set (see also \cite{lam2001first}). The failure of commutativity (and associativity in the octonions $\OO$) induces a dimensional inflation of the root set. For central polynomials (coefficients in $\R$), 0-dimensional sets in $\C$ inflate into continuous manifolds ($S^2$ in $\HH$, $S^6$ in $\OO$). The associated symmetry group transitions from a discrete group to a compact Lie group ($SO(3)$ or $G_2$), arising from the action of the automorphism group.

This paper introduces a dynamical framework, Dynamical Non-Commutative Algebraic Geometry (DNCAG), where we treat the polynomial coefficients as parameters. We move beyond the static characterization of root sets to analyze their evolution, stability, and bifurcations. While the underlying algebraic structures are established, our approach utilizes tools from dynamical systems (cf. \cite{guckenheimer1983nonlinear}), algebraic deformation theory (cf. \cite{gerstenhaber1964deformation}), and Morse-Bott theory (cf. \cite{bott1954nondegenerate}) to analyze the stability of these structures under perturbations.

The motivation for developing DNCAG extends beyond pure mathematics into theoretical physics. The geometric structures analyzed here—spherical manifolds arising from algebraic constraints—are ubiquitous. In quaternionic quantum mechanics \cite{adler1995quaternionic}, these manifolds may represent parameter spaces or emergent symmetries. The dynamical collapse of topology offers a novel algebraic model for phenomena involving dimensional reduction and symmetry breaking, such as those hypothesized near black hole singularities or characterizing transitions between topological phases of matter. Furthermore, the extension to octonions is motivated by their potential role in unified theories and their unique algebraic properties \cite{baez2002octonions}.

Our main original contributions are:
\begin{enumerate}
    \item The analysis of central dynamics (breathing modes), including spectral characterization of non-linear coupling and the classification of dynamics near central bifurcations (transversal vs. tangential crossings).
    \item The formalization of the topological collapse dynamics using gradient flow, including the proof of critical slowing down and the analysis of basin decomposition.
    \item The introduction of a thermodynamic formalism (Gibbs measure, Entropy Scaling Law, Order Parameter) to characterize the algebraic collapse as a statistical phase transition.
\end{enumerate}

\subsection{Main Results}

The central contribution of this paper is the characterization of the topological instability induced by non-central perturbations across division algebras.

\begin{definition}[Analytic Deformation]\label{def:analytic_deformation}
Let $\A$ be a real division algebra. Let $P_0(x) \in \R[x]$ be a central polynomial. Let $G_P$ be the automorphism stabilizer of $P$ (Definition~\ref{def:aut_stabilizer}). A generically non-central analytic deformation is a family of polynomials $P_\epsilon(x) \in \A[x]$ parameterized by $\epsilon \in \R$, such that the coefficients depend analytically on $\epsilon$, $P_\epsilon(x) = P_0(x)$ when $\epsilon=0$, and $G_{P_\epsilon}$ is trivial for $\epsilon \neq 0$ in a punctured neighborhood of $\epsilon=0$.
\end{definition}

\begin{maintheorem}[Dynamical Non-Commutative Bifurcation]\label{thm:main_bifurcation}
Let $P_0(x)$ be a central polynomial over $\A \in \{\HH, \OO\}$ such that its root variety $V_0 = Z(P_0)$ contains a non-real manifold of dimension $d_{\rm M} > 0$. Let $P_\epsilon(x)$ be a generically non-central analytic deformation. Let $V_\epsilon = Z(P_\epsilon)$. Let $d_{\rm A} = \dim(\A)$.
\begin{enumerate}
    \item \textbf{Topological Collapse:} The Hausdorff dimension of the root variety changes discontinuously at $\epsilon=0$.
    \item \textbf{Algebraic Singularity:} The central fiber $V_0$ is characterized by a Jacobian rank deficiency: $\rank(J_{P_0}(x)) = d_{\rm A} - d_{\rm M}$.
    \item \textbf{Dynamical Retraction:} Let $\mathcal{V}_\epsilon(x) = \|P_\epsilon(x)\|^2$. The gradient flow $\dot{x} = -\nabla\mathcal{V}_\epsilon(x)$ realizes the collapse as a deformation retract. The timescale exhibits critical slowing down: $T_{\rm collapse} \propto \epsilon^{-2}$.
    \item \textbf{Thermodynamic Phase Transition:} In a statistical ensemble at temperature $T$, the collapse manifests as a phase transition. An alignment order parameter $m(\epsilon,T)$ exhibits a discontinuity at $(\epsilon, T) = (0, 0)$.
\end{enumerate}
\end{maintheorem}

\begin{proof}
We synthesize the results established in the subsequent sections.

\textbf{Part 1 (Topological Collapse):} For $\epsilon=0$, by the Dimensional Inflation Law (Corollary~\ref{cor:inflation_law}), $d_{\rm M} = d_{\rm A} - 2$. For $\epsilon \neq 0$, $G_{P_\epsilon}$ is trivial. By the Generalized Symmetry Reduction Theorem (Theorem~\ref{thm:generalized_symmetry_reduction}), the roots are isolated ($d_{\rm M}=0$). The discontinuity follows (Proposition~\ref{prop:collapse}).

\textbf{Part 2 (Algebraic Singularity):} The tangent space $T_x V_0$ is related to $\Ker(J_{P_0}(x))$. The rank deficiency follows from the positive dimension of the manifold (Theorem~\ref{thm:jacobian_singularity}).

\textbf{Part 3 (Dynamical Retraction):} The potential $\mathcal{V}_\epsilon(x)$ is real-analytic and coercive. The Lojasiewicz gradient inequality \cite{lojasiewicz1963propriete} guarantees convergence of the gradient flow (Theorem~\ref{thm:deformation_retract}). The timescale scaling $T_{\rm collapse} \propto \epsilon^{-2}$ is derived from the quadratic scaling of the restricted potential (Theorem~\ref{thm:collapse_timescale}) and the generic Lojasiewicz exponent (Proposition~\ref{prop:lojasiewicz_exponent}).

\textbf{Part 4 (Thermodynamic Phase Transition):} We define the Gibbs measure (Definition~\ref{def:gibbs_measure}) and the alignment order parameter $m(\epsilon, T)$ (Definition~\ref{def:order_parameter}). In the limit $T \to 0$, the measure concentrates on $V_\epsilon$. For $\epsilon=0$, symmetry implies $m(0, 0) = 1/(d_{\rm A}-1)$. For $\epsilon \neq 0$, the Localization Theorem (Theorem~\ref{thm:algebraic_alignment}) implies alignment, $m(\epsilon, 0) = 1$. The discontinuity characterizes the phase transition (Theorem~\ref{thm:phase_transition}).
\end{proof}

\section{Algebraic Structure and Dimensional Inflation}

\subsection{The Baseline: Rigid Geometry in $\C$}

In $\C$, the geometry of the roots of a lacunary polynomial is constrained by commutativity, resulting in discrete rotational symmetry ($C_d$).

\begin{theorem}[Geometric Rigidity in $\C$]\label{thm:symmetry_gcd}
Let $P(z) \in \C[z]$ be lacunary. Let $d = \gcdOp(\{\text{exponents}\})$. If $d > 1$, the set of roots $Z(P)$ is invariant under the action of the cyclic group $C_d$.
\end{theorem}

\subsection{Inflation in Division Algebras and Group Actions}

In non-commutative division algebras, the $C_d$ symmetry is generally broken. For central polynomials, a continuous symmetry emerges, governed by the algebra's automorphisms. We assume the standard polynomial definition $P(x) = \sum a_k x^k$, which is well-defined as $\HH$ and $\OO$ are power-associative.

\begin{definition}[Automorphism Group]
Let $\A$ be an algebra over $\R$. The automorphism group $\Aut(\A)$ is the group of all bijective linear maps $g: \A \to \A$ preserving multiplication: $g(xy) = g(x)g(y)$.
\end{definition}

\begin{theorem}[Generalized Inflation Theorem]\label{thm:generalized_inflation}
Let $\A$ be a finite-dimensional real division algebra. Let $P(x) \in \R[x]$. For any non-central root $\alpha \in \A \setminus Z(\A)$, the solution set $Z(P)$ contains a submanifold diffeomorphic to the homogeneous space $G/H$, where $G = \Aut(\A)$, and $H = \Stab_G(\alpha)$.
\end{theorem}

\begin{proof}
The automorphism group $G$ acts on $\A$. Since the coefficients of $P$ are real (central), they are fixed by $G$. If $\alpha$ is a root, then for any $g \in G$, $P(g(\alpha)) = g(P(\alpha)) = 0$. The orbit of $\alpha$ under $G$, $O(\alpha)$, is contained in $Z(P)$. By the Orbit-Stabilizer Theorem, $O(\alpha)$ is a smooth manifold diffeomorphic to $G/H$.
\end{proof}

\begin{corollary}[Dimensional Inflation Law]\label{cor:inflation_law}
Let $\A$ be a real normed division algebra ($\C, \HH, \OO$) of dimension $d_{\rm A} \ge 2$. The dimension of the root manifold $M_{x_0}$ is $d_{\rm M} = d_{\rm A} - 2$.
\end{corollary}

\begin{proof}
The root manifolds are spheres $S^{d_{\rm A}-2}$.
$\C$ ($d_{\rm A}=2, G=C_2, d_{\rm M}=0$). $\HH$ ($d_{\rm A}=4, G=SO(3), d_{\rm M}=2$). $\OO$ ($d_{\rm A}=8, G=G_2, d_{\rm M}=6$).
\end{proof}

\begin{remark}[Associativity and Conjugation]
In $\HH$ (associative), the Skolem-Noether theorem implies $\Aut(\HH)$ consists of inner automorphisms (conjugations $q \mapsto hqh^{-1}$). Thus, the orbit $O(q_0)$ coincides with the conjugacy class $[q_0]$. In $\OO$ (non-associative), $\Aut(\OO)=G_2$ is strictly larger than the group generated by conjugations. Theorem~\ref{thm:generalized_inflation} correctly identifies the root manifold using the full automorphism group $G_2$.
\end{remark}

\section{Dynamical Geometry: Central Modulation}

We introduce a dynamical framework by considering polynomials $P(x, t) = \sum a_k(t) x^k$. This analysis relies only on power-associativity.

\subsection{Stability, Bifurcation, and Breathing Modes}

We analyze the trinomial system $P(x,t) = x^{2k} + a(t)x^k + b(t) = 0$. The auxiliary roots depend on the discriminant $\Delta(t) = a(t)^2 - 4b(t)$.

The stability region for distinct spheres of purely imaginary roots requires $\Delta(t) > 0$, $b(t)>0$, and $a(t)<0$.

\begin{definition}[Breathing Mode]
If the coefficients $a_k(t)$ remain central and satisfy the stability conditions, the root manifolds undergo continuous metric deformation while preserving their topological class.
\end{definition}

\begin{proposition}[Classification of Central Topological Bifurcations]\label{prop:central_bifurcations}
Under central modulation ($a(t), b(t) \in \R$), the topology of the root set changes when the trajectory crosses the stability boundaries (Saddle-Node, Pitchfork, Hopf bifurcations).
\end{proposition}

Figure~\ref{fig:breathing_modes} illustrates the dynamic behavior within the stability region for $k=2$ in $\HH$.

\begin{figure}[htbp]
    \centering
    \includegraphics[width=\textwidth]{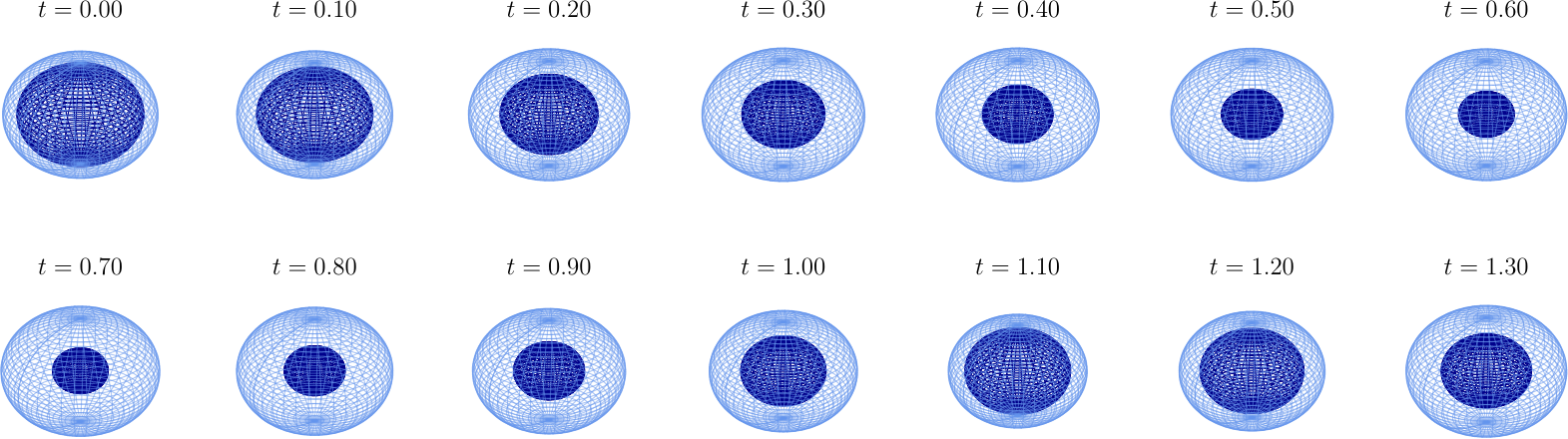}
    \caption{Dynamics of Quaternionic Root Manifolds (Breathing Modes). Visualization of the solution set for $P(q,t) = q^4 + a(t)q^2 + b(t)$ in $\HH$. The system exhibits coupled radial oscillations and changes in the radial separation ratio. An animation of these modes can be seen in \url{http://youtu.be/WJTjyMth1FI}}
    \label{fig:breathing_modes}
\end{figure}

\subsection{Non-Linear Dynamics and Spectral Analysis}

In the stable regime, the radii $R_{\rm inner}, R_{\rm outer}$ are coupled non-linearly:
\begin{equation}\label{eq:radii_explicit}
R_{\rm inner, outer}(t) = \sqrt{\frac{|a(t)| \mp \sqrt{\Delta(t)}}{2}}.
\end{equation}
\noindent Sinusoidal driving generates harmonics and intermodulation frequencies (Figure~\ref{fig:spectral_analysis}).

\begin{figure}[htbp]
    \centering
    \includegraphics[width=0.95\textwidth]{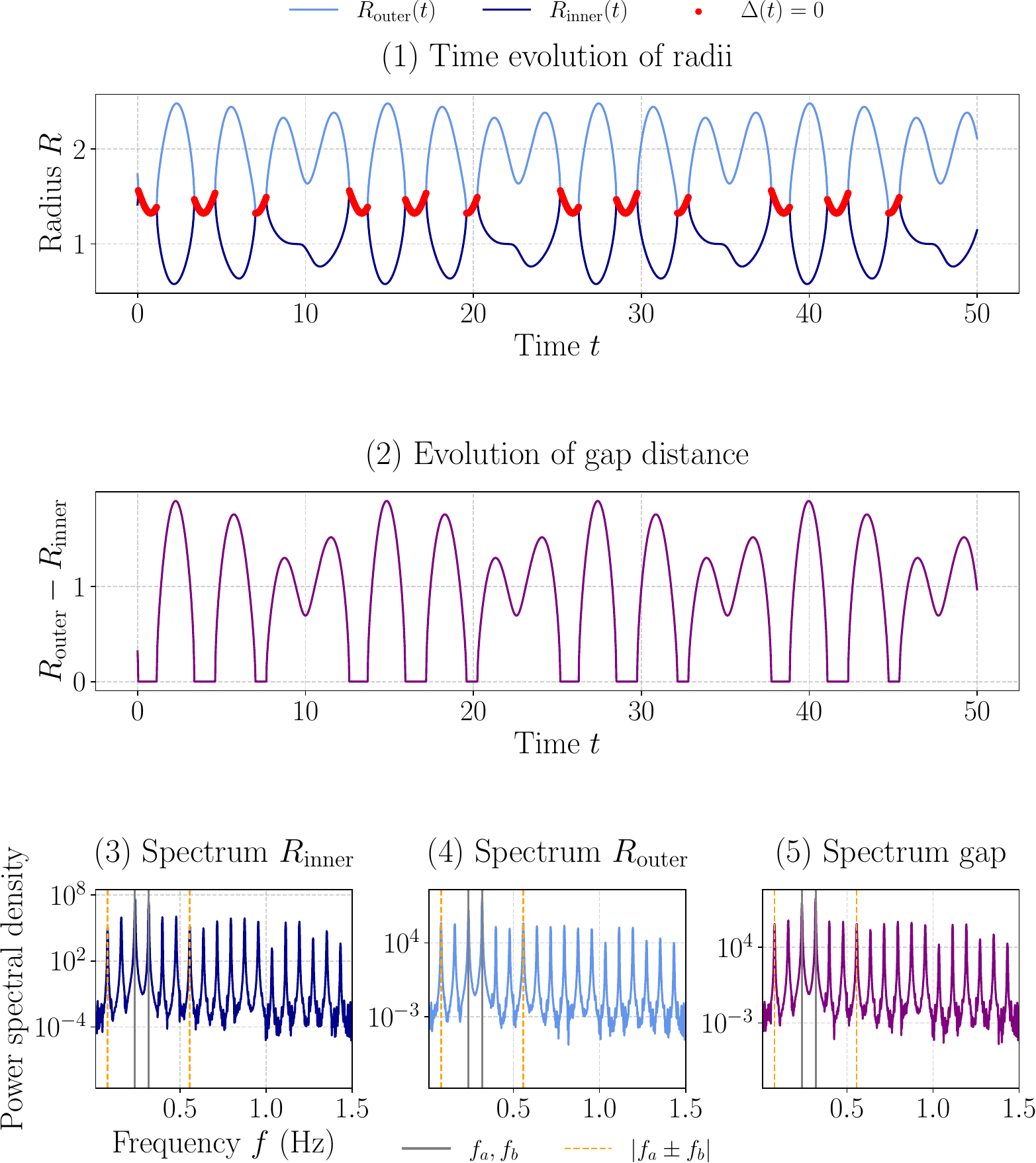}
    \caption{Spectral Analysis of Breathing Modes (in $\HH$). (1) Time evolution of radii. (2) Evolution of the Discriminant $\Delta(t)$. Markers indicate proximity to the bifurcation boundary, leading to sharp transitions (cusps). (3) Evolution of the gap distance. (4-6) Power spectra (PSD), confirming strong non-linear coupling.}
    \label{fig:spectral_analysis}
\end{figure}

\subsection{Dynamics Near Topological Bifurcation}\label{sec:dynamics_near_bifurcation}

The dynamics near $\Delta(t) = 0$ exhibit distinct behaviors depending on the velocity of the discriminant at the crossing point $t_{\rm c}$.
\begin{equation}\label{eq:velocity_radii}
\frac{d}{dt}(R^2_{1,2}(t)) \propto \mp \frac{\dot{\Delta}(t)}{2\sqrt{\Delta(t)}}.
\end{equation}

\begin{enumerate}
    \item \textbf{Transversal Crossing (``Pushing''):} $\dot{\Delta}(t_{\rm c}) \neq 0$. The velocities become singular (infinite) at $t_{\rm c}$.
    \item \textbf{Tangential Crossing (``Hovering''):} $\dot{\Delta}(t_{\rm c}) = 0$. The singularity is removable, and velocities remain finite.
\end{enumerate}

\section{Topological Collapse: Symmetry Reduction}

We analyze the dynamics when the coefficients $a_k$ leave the center of the algebra $\A$.

\begin{proposition}[Topological Collapse]\label{prop:collapse}
For a generically non-central perturbation $P_\epsilon(x)$, the root manifold undergoes a topological phase transition. The Hausdorff dimension of the solution set $Z(P_\epsilon)$ changes discontinuously at $\epsilon=0$. (Part 1 of Main Theorem~\ref{thm:main_bifurcation}).
\end{proposition}

The collapse is governed by the reduction of the symmetry group acting on the roots. To accommodate non-associative algebras ($\OO$), we must use the automorphism group.

\begin{definition}[Automorphism Stabilizer of the Polynomial]\label{def:aut_stabilizer}
Let $P(x) \in \A[x]$. The automorphism stabilizer of $P$ is the subgroup of $\Aut(\A)$ that fixes all coefficients:
\begin{equation}
G_P = \{g \in \Aut(\A) \mid g(a_k) = a_k \text{ for all } k\}.
\end{equation}
\end{definition}

\begin{theorem}[Generalized Symmetry Reduction Theorem]\label{thm:generalized_symmetry_reduction}
The root set $Z(P)$ is invariant under the action of $G_P$. The geometry of the root manifolds is determined by the orbits of $G_P$.
\end{theorem}

\begin{proof}
If $g \in G_P$, then $P(g(x)) = \sum g(a_k) g(x)^k = \sum a_k g(x)^k$. If $P(\alpha)=0$, then $P(g(\alpha))=0$.
\end{proof}

\subsection{The Localization Theorem (Alignment Principle)}

We now address the principle that the resulting roots align with the algebraic structure of the coefficients. This result, established by Gordon and Motzkin \cite{gordon1965zeros}, is robust across all real normed division algebras, as it relies on alternativity rather than associativity.

\begin{definition}[Coefficient Subalgebra]
Let $P(x) \in \A[x]$. The coefficient subalgebra $\A(P)$ is the smallest $\R$-subalgebra of $\A$ containing all coefficients $a_k$.
\end{definition}

\begin{theorem}[Localization Theorem (Gordon-Motzkin)]\label{thm:algebraic_alignment}
Let $\A$ be a real normed division algebra ($\C, \HH, \OO$). Let $P(x) \in \A[x]$. Any isolated root $x_0$ of $P(x)$ belongs to the coefficient subalgebra $\A(P)$.
\end{theorem}

\begin{proof}
Let $x_0$ be an isolated root. Let $M(x) = x^2 - 2\Real(x_0)x + \|x_0\|^2$ be the minimal polynomial of $x_0$ over $\R$.

We apply the right division algorithm in $\A[x]$ to divide $P(x)$ by $M(x)$. Since $M(x)$ is central and $\A$ is alternative (which holds for $\C, \HH, \OO$ \cite{schafer1995introduction}), the division is valid:
\begin{equation}
P(x) = Q(x) M(x) + R(x).
\end{equation}
The remainder is $R(x) = A x + B$, with $A, B \in \A(P)$.

Evaluating at $x=x_0$, we have $A x_0 + B = 0$.

If $A=0$, then $B=0$. This implies $M(x)$ divides $P(x)$. The entire orbit $O(x_0)$ (e.g., $S^2$ or $S^6$) are roots (as $M(x)$ is central). This contradicts the assumption that $x_0$ is an isolated root (unless $x_0$ is real).

Therefore, $A \neq 0$. $\A(P)$ is a finite-dimensional subalgebra of $\A$. Since $\A$ is an alternative division algebra, $\A(P)$ is also an alternative division algebra. Thus, $A$ is invertible in $\A(P)$.

We solve for $x_0$: $x_0 = -A^{-1} B$. Since $A^{-1}, B \in \A(P)$, we conclude $x_0 \in \A(P)$.
\end{proof}

\begin{remark}[Genericity Conditions]\label{rem:genericity_alignment}
The theorem requires that the remainder coefficient $A \neq 0$. This condition connects the algebraic property of the coefficients to the topological property of the root being isolated. It fails only in nongeneric cases where $M(x)$ divides $P(x)$, which occurs when the root $x_0$ remains part of a continuous manifold (i.e., the symmetry reduction is incomplete). In the generic non-central case defining isolated roots, $A \neq 0$.
\end{remark}

\begin{corollary}[Alignment Principle]\label{cor:alignment_principle}
If the coefficients of $P(x)$ all belong to a proper subalgebra (e.g., a copy of $\C$ or $\HH$ within $\OO$), then all isolated roots must also belong to that subalgebra.
\end{corollary}

\section{The Dynamics of Collapse: Deformation and Flow}

We analyze the topological collapse using algebraic geometry (deformation theory) and differential geometry (Morse-Bott theory and gradient flow). These methods rely on the smoothness of the potential landscape and generalize directly to $\OO$.

\subsection{Deformation Theory and Jacobian Singularity}

We consider a deformation $P_\epsilon(x)$. The total space is $\mathcal{X} = \{ (x, \epsilon) \mid P_\epsilon(x) = 0 \}$. The deformation is non-flat as $\dim(V_\epsilon) < \dim(V_0)$.

\begin{theorem}[Jacobian Singularity of the Central Fiber]\label{thm:jacobian_singularity}
Let $V_0$ be the root manifold of a central polynomial $P_0$. The Jacobian $J_{P_0}(x)$ (viewed as a map $\R^{d_{\rm A}} \to \R^{d_{\rm A}}$) is singular for all $x \in V_0$. (Part 2 of Main Theorem~\ref{thm:main_bifurcation}).
\end{theorem}

\begin{proof}
The tangent space $T_x V_0$ is related to $\Ker(J_{P_0}(x))$. Since $\dim(V_0) = d_{\rm M} > 0$, by the rank-nullity theorem, $\rank(J_{P_0}(x)) = d_{\rm A} - d_{\rm M}$.
\end{proof}

\subsection{Morse-Bott Theory and Gradient Flow Dynamics}

We connect the algebraic singularity to the geometry of the associated potential landscape.

\begin{definition}[Potential Landscape]
Define the real-analytic function $\mathcal{V}_\epsilon: \A \to \R_{\ge 0}$ as $\mathcal{V}_\epsilon(x) = \|P_\epsilon(x)\|^2$.
\end{definition}

\begin{proposition}[Morse-Bott Degeneracy]
The Jacobian singularity corresponds to the Morse-Bott degeneracy of $\mathcal{V}_0$. The Hessian $\Hess(\mathcal{V}_0)$ is zero along directions tangent to $V_0$.
\end{proposition}

\noindent The perturbation $\mathcal{V}_\epsilon$ lifts this degeneracy. We study the transition dynamics using the gradient flow.

\begin{definition}[Gradient Flow]\label{def:gradient_flow}
The gradient flow is the dynamical system defined by
\begin{equation}\label{eq:gradient_flow}
\dot{x}(t) = -\nabla \mathcal{V}_\epsilon(x(t)).
\end{equation}
\end{definition}

\subsection{Deformation Retract and Basins of Attraction}

The evolution of $V_0$ under the gradient flow describes the dynamics of the collapse.

\begin{theorem}[Deformation Retract]\label{thm:deformation_retract}
The gradient flow defines a deformation retract from the initial manifold $V_0$ onto the final set of attractors $V_\epsilon$. (Part 3 of Main Theorem~\ref{thm:main_bifurcation}).
\end{theorem}

\begin{proof}
Since $\mathcal{V}_\epsilon$ is real-analytic and coercive (in any normed division algebra), the Lojasiewicz gradient inequality \cite{lojasiewicz1963propriete} guarantees that every trajectory converges to a critical point as $t\to \infty$. The global minima $V_\epsilon$ are the stable attractors. The flow map $\Phi_t^\epsilon: \A \to \A$ defines a continuous evolution. Restricting the flow to the union of the basins of attraction of $V_\epsilon$ provides the deformation retract.
\end{proof}

\begin{definition}[Stable Manifold (Basin of Attraction)]
Let $x^*_i \in V_\epsilon$. The stable manifold $W^s(x^*_i)$ is the set of points in $\A$ that flow to $x^*_i$.
\end{definition}

\begin{theorem}[Basin Decomposition]
$V_0 = \bigcup_{i} (V_0 \cap W^s(x^*_i))$.
\end{theorem}

\begin{figure}[htbp]
    \centering
    \includegraphics[width=\textwidth]{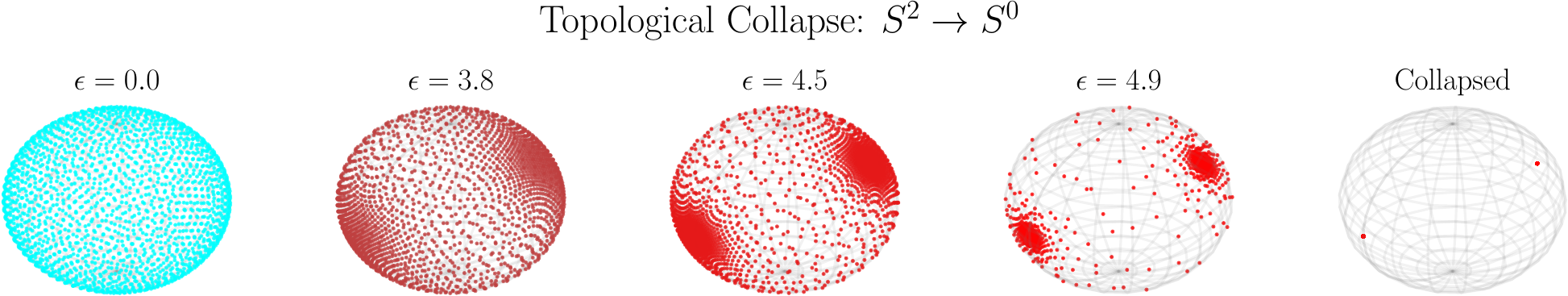}
    \caption{Visualization of Topological Collapse via Gradient Flow (in $\HH$). The sequence illustrates the transition from $S^2$ to $S^0$ under a non-central perturbation aligned with the $i$-axis. (1) Initial state $V_0$. (2-4) Evolution under the gradient flow $-\nabla\mathcal{V}_\epsilon$. (5) Final state $V_\epsilon$: The manifold has collapsed onto the isolated roots (Alignment Principle). An animation of the collapse is shown in \url{http://youtu.be/yaJQOuftZjE}}
    \label{fig:topological_collapse}
\end{figure}

Figure~\ref{fig:topological_collapse} visualizes the streamlines of the gradient flow in $\HH$.

\section{The Dynamics of Collapse: Gradient Flow Analysis}

We analyze the relaxation process using the gradient flow and quantify its timescale.

\subsection{The Potential Landscape and Degeneracy Lifting}

The dynamics of the transition are governed by the restricted potential $f_\epsilon: V_0 \to \R_{\ge 0}$, $f_\epsilon(x) = \mathcal{V}_\epsilon(x)|_{V_0}$.

\begin{proposition}\label{prop:morse_lifting}
For a generic non-central perturbation, the restricted potential $f_\epsilon$ is a Morse function on the manifold $V_0$.
\end{proposition}

\subsection{Gradient Flow Dynamics}

The initial flow on the manifold is approximated by the Riemannian gradient $\nabla_{\rm M}$:
\begin{equation}\label{eq:gradient_flow_manifold}
\dot{x}(t) \approx -\nabla_{\rm M} f_\epsilon(x(t)).
\end{equation}

\subsection{Analysis of the Canonical Perturbation (in $\HH$)}

We illustrate this with the canonical example $P_1(q) = q^2+iq+1$ in $\HH$. $V_0=S^2$, $V_1 = \{i, -i\}$.
The restricted potential is $f_1(\phi) = \sin^2\phi$ (polar angle $\phi$ from the $i$-axis).

The Riemannian gradient flow on $S^2$ is $\dot{\phi} = -\sin(2\phi)$.

\subsection{Basins of Attraction}

\begin{theorem}[Basin Decomposition in $\HH$]
For the canonical perturbation $P_1$ in $\HH$, the initial manifold $V_0$ is decomposed into the Northern hemisphere ($W^s(i)$) and the Southern hemisphere ($W^s(-i)$).
\end{theorem}

\begin{figure}[htbp]
    \centering
    \includegraphics[width=0.55\textwidth]{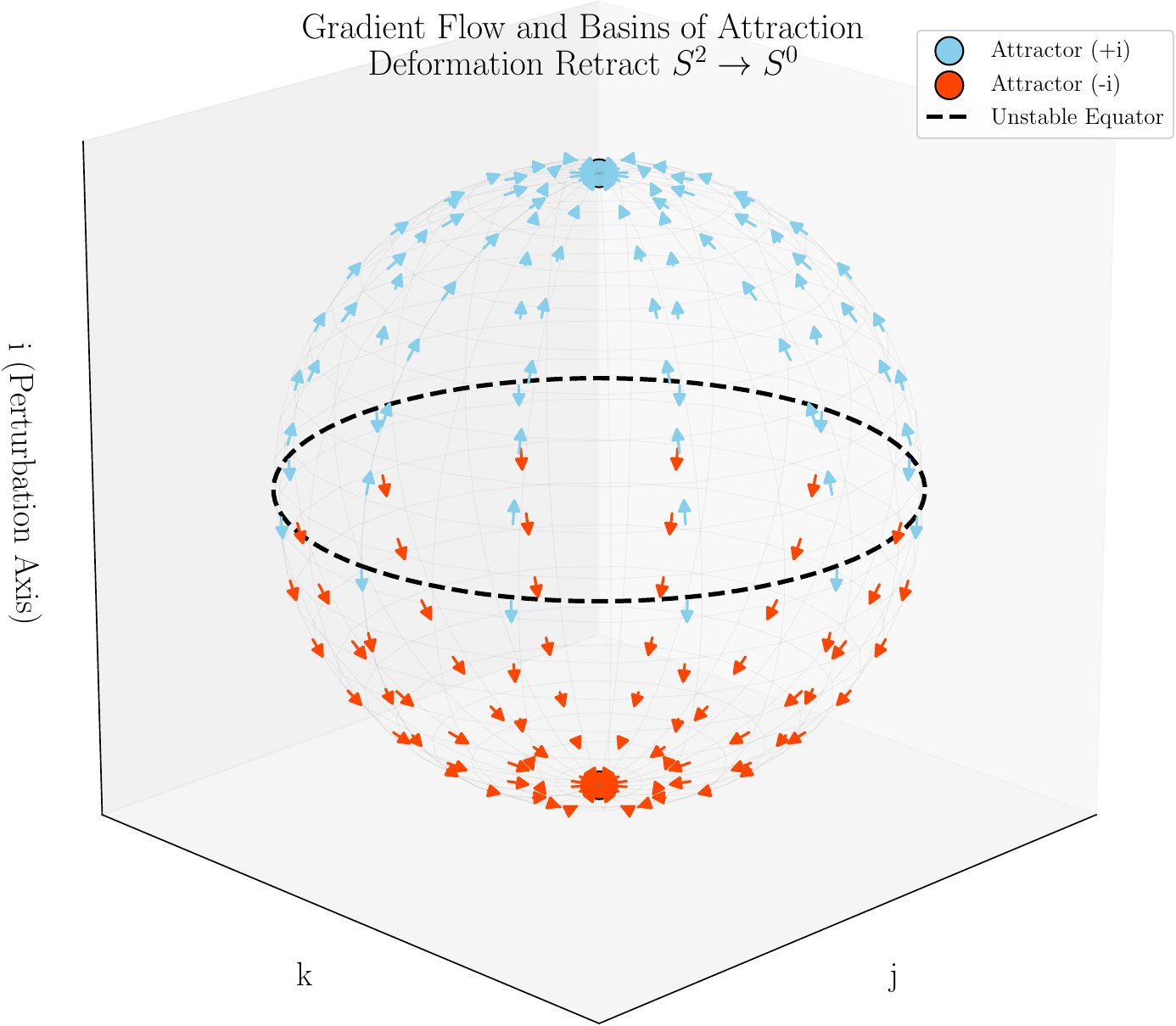}
    \caption{Geometric Decomposition of the Topological Collapse in $\HH$. The plot visualizes the gradient flow on $V_0 \cong S^2$. The equator acts as an unstable equilibrium separatrix, dividing the manifold into two basins of attraction flowing to the attractors at $\pm i$.}
    \label{fig:gradient_flow}
\end{figure}

Figure \ref{fig:gradient_flow} maps the streamlines of the gradient flow on the spherical manifold, illustrating this decomposition.

\subsection{Quantification of Collapse Timescales and Critical Slowing Down}\label{sec:critical_slowing_down}

We quantify the timescale of the collapse by analyzing the convergence rate of the gradient flow on the manifold. We consider a small perturbation $P_\epsilon(x) = P_0(x) + \epsilon \Delta P(x)$.

\begin{lemma}[Potential Scaling]\label{lemma:potential_scaling}
The restricted potential scales quadratically with the perturbation magnitude: $f_\epsilon(x) = \epsilon^2 \|\Delta P(x)\|^2|_{V_0} + O(\epsilon^3)$.
\end{lemma}

\begin{proof}
$\mathcal{V}_\epsilon(x) = \|P_0(x) + \epsilon \Delta P(x)\|^2$. On $V_0$, $P_0(x)=0$. Thus $f_\epsilon(x) = \|\epsilon \Delta P(x)\|^2 = \epsilon^2 \|\Delta P(x)\|^2$.
\end{proof}

The gradient flow on the manifold (Eq.~\ref{eq:gradient_flow_manifold}) therefore scales as $\dot{x}(t) \propto -\epsilon^2$.

\begin{theorem}[Critical Slowing Down (Quadratic Scaling)]\label{thm:collapse_timescale}
The characteristic relaxation time $T_{\rm collapse}$ for the topological collapse under a generic non-central perturbation of magnitude $\epsilon$ scales as $T_{\rm collapse} \propto 1/\epsilon^2$.
\end{theorem}

\begin{proof}
We analyze the flow near the stable attractor $x^*$. Since $f_\epsilon \propto \epsilon^2$ (Lemma~\ref{lemma:potential_scaling}), the gradient $\nabla_{\rm M} f_\epsilon \propto \epsilon^2$.

We consider the linearization of the flow near the attractor. Let $y$ be the local coordinate on the manifold. The linearized dynamics near the minimum are $\dot{y} \approx -\lambda \epsilon^2 y$, where $\lambda$ is related to the curvature of the restricted potential $\|\Delta P(x)\|^2$.
The solution is $y(t) \approx y(0) \exp(-\lambda \epsilon^2 t)$.
The characteristic relaxation time is $T_{\rm collapse} = 1/(\lambda \epsilon^2)$.
\end{proof}

\begin{proposition}[Lojasiewicz Exponent]\label{prop:lojasiewicz_exponent}
For a generic analytic perturbation (Morse function $f_\epsilon$), the Lojasiewicz exponent characterizing the flatness of the potential near the critical manifold $V_0$ is $\theta = 1/2$. This yields the quadratic timescale $T_{\rm collapse} \propto \epsilon^{-2}$. At non-generic perturbations (where the restricted potential is flatter than quadratic), $\theta$ may be smaller ($0 < \theta < 1/2$), leading to slower collapse timescales $T_{\rm collapse} \propto \epsilon^{-(1+\theta)/\theta}$.
\end{proposition}

\noindent This phenomenon is known as Critical Slowing Down. As $\epsilon \to 0$, the dynamics of the collapse become arbitrarily slow.

\section{Statistical Geometry and the Thermodynamic Formalism}

We introduce a statistical framework to study the global properties of the potential landscape $\mathcal{V}(x)$. We introduce a parameter $T \ge 0$ (temperature) which models the level of noise.

\subsection{The Statistical Ensemble and the Gibbs Measure}

\begin{definition}[Gibbs Measure and Partition Function]\label{def:gibbs_measure}
The Gibbs probability measure $\mu_T$ on $\A$ is defined by
\begin{equation}
d\mu_T(x) = \frac{1}{Z(T)} \exp\left(-\frac{\mathcal{V}(x)}{T}\right) dx,
\end{equation}
where $Z(T) = \int_{\A} \exp(-\mathcal{V}(x)/T) dx$ is the partition function.
\end{definition}

\begin{proposition}[Zero Temperature Limit]\label{prop:zero_temp_limit}
As $T \to 0^+$, $\mu_T$ converges weakly to a measure concentrated uniformly on the variety of roots $V(P)$.
\end{proposition}

\subsection{Free Energy, Entropy, and Geometric Degeneracy}

The Helmholtz Free Energy is $F(T) = -T \ln Z(T)$. The Entropy is $S(T) = -\partial F / \partial T$.

In the low temperature limit ($T \to 0^+$), $Z(T)$ is dominated by the geometry of $V(P)$. We analyze this using Laplace's method.

\begin{proposition}[Entropy Scaling Law]\label{prop:entropy_scaling}
The asymptotic behavior of the entropy $S(T)$ as $T \to 0^+$ depends on the dimension of the root manifold $d_{\rm M}$. Let $d_{\rm A}=\dim(\A)$.
\begin{enumerate}
    \item \textbf{Degenerate Manifolds (Central, $d_{\rm M}>0$):}
    \begin{equation}
    S(T) = \frac{d_{\rm A}-d_{\rm M}}{2}\log T + \log\Vol(V(P)) + O(1).
    \end{equation}
    \item \textbf{Isolated Roots (Non-Central, $d_{\rm M}=0$):}
    \begin{equation}
    S(T) = \frac{d_{\rm A}}{2}\log T + O(1).
    \end{equation}
\end{enumerate}
\end{proposition}

\begin{proof}[Proof Sketch]
We analyze the asymptotic behavior of $Z(T)$ using Laplace's method.

\textbf{Case 1 ($d_{\rm M}=0$).} Morse minima. $Z(T) \propto T^{d_{\rm A}/2}$.

\textbf{Case 2 ($d_{\rm M}>0$).} Morse-Bott minima. $Z(T) \propto \Vol(V(P)) \cdot T^{(d_{\rm A}-d_{\rm M})/2}$.

\textbf{Entropy Calculation:}
Let $Z(T) \approx C T^\alpha$, where $\alpha = (d_{\rm A}-d_{\rm M})/2$.
$S(T) = \ln Z(T) + T(Z'(T)/Z(T)) \approx \alpha \log T + \log C + \alpha$. The leading term behavior and the constant terms establish the scaling law.
\end{proof}

\begin{corollary}[Collapse as Entropy Reduction]
The topological collapse ($d_{\rm M} \to 0$) corresponds to a reduction in the entropy of the system in the low temperature limit.
\end{corollary}

\subsection{Topological Collapse as a Phase Transition}

We analyze the topological collapse as a phase transition characterized by symmetry breaking.

\begin{definition}[Alignment Order Parameter]\label{def:order_parameter}
Let $\nu$ be a unit imaginary vector defining the perturbation axis. The alignment order parameter $m_\nu$ is the normalized mean square alignment of the coordinate along $\nu$ under the Gibbs measure. For $\HH$ ($d_{\rm A}=4$) aligned along $i$ (coordinate $b$):
\begin{equation}
m(\epsilon, T) = \frac{\langle b^2 \rangle_{\mu_T(\epsilon)}}{\langle b^2+c^2+d^2 \rangle_{\mu_T(\epsilon)}}.
\end{equation}
\end{definition}

\begin{theorem}[Symmetry Breaking and Phase Transition]\label{thm:phase_transition}
Consider a perturbation $P_\epsilon(x)$ breaking the symmetry along the $\nu$-axis. We analyze the behavior in the zero temperature limit ($T \to 0$). (Part 4 of Main Theorem~\ref{thm:main_bifurcation}).
\begin{enumerate}
    \item \textbf{Symmetric Phase (Disordered):} For $\epsilon=0$, the measure concentrates uniformly on $S^{d_{\rm A}-2}$. $m(0, 0) = 1/(d_{\rm A}-1)$.
    \item \textbf{Ordered Phase (Aligned):} For $\epsilon \neq 0$, the measure concentrates on the isolated roots aligned with the $\nu$-axis. $m(\epsilon, 0) = 1$.
\end{enumerate}
\end{theorem}

\begin{proof}
As $T \to 0$, the measure concentrates on $V(P)$.
For $\epsilon=0$, by symmetry, the expectation of the square of any imaginary coordinate is equal. There are $d_{\rm A}-1$ imaginary dimensions. Thus $m(0,0)=1/(d_{\rm A}-1)$. (In $\HH$, $1/3$; in $\OO$, $1/7$).
For $\epsilon \neq 0$, the Localization Theorem (Theorem~\ref{thm:algebraic_alignment}) implies the ground state localizes on the axis defined by the subalgebra (e.g., the $\nu$-axis), so $m(\epsilon, 0)=1$.
\end{proof}

\noindent The topological collapse is characterized as a phase transition at $T=0$, where the order parameter changes discontinuously at $\epsilon=0$.

\begin{proposition}[Smoothing and Statistical Symmetry Restoration]
At $T>0$, the sharp phase transition is replaced by a smooth crossover. If $T \gg \epsilon$, the noise dominates, $m(\epsilon, T) \approx 1/(d_{\rm A}-1)$. The symmetry is statistically restored by thermal fluctuations.
\end{proposition}

\begin{figure}[htbp]
    \centering
    \includegraphics[width=\textwidth]{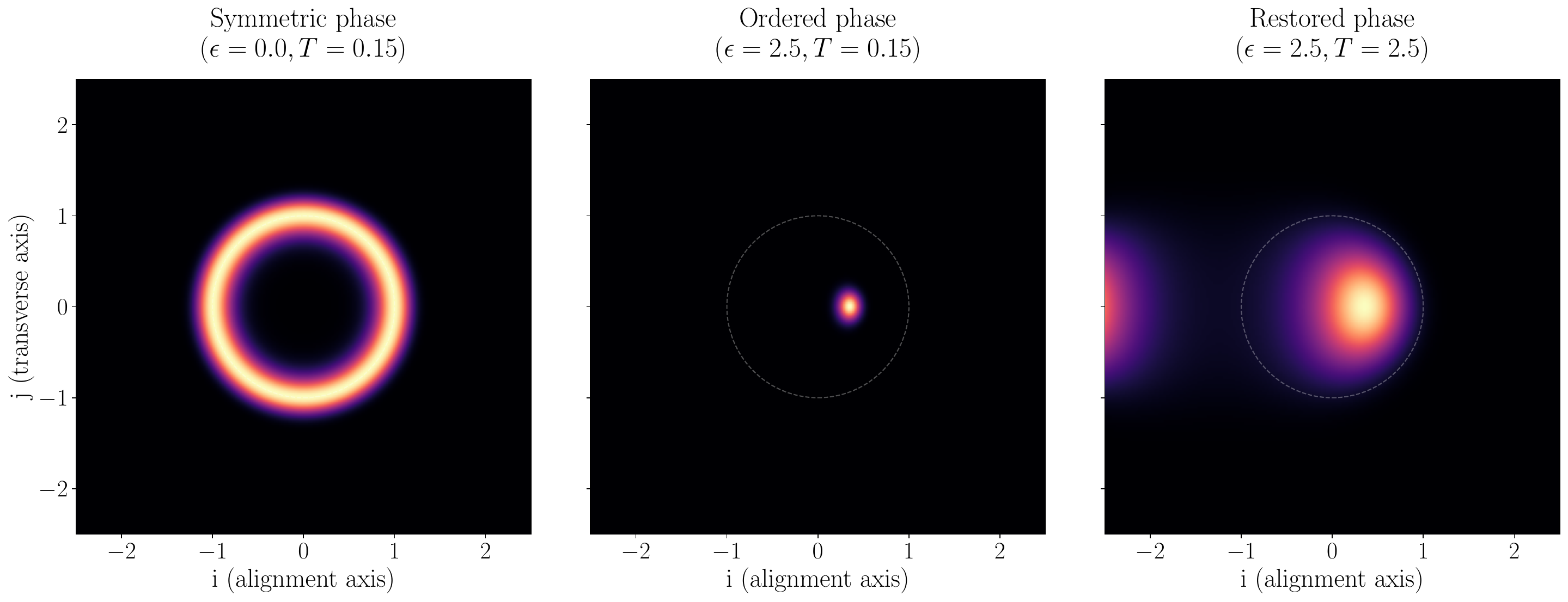}
    \caption{Thermodynamic Phase Transition and Statistical Symmetry Restoration (in $\HH$). Visualization of the Gibbs probability measure $\mu_T$. (Left) Symmetric Phase ($\epsilon=0$). (Center) Ordered Phase ($\epsilon=2.5, T=0.15$). (Right) Restored Phase ($\epsilon=2.5, T=2.50$).}
    \label{fig:thermal_analysis}
\end{figure}

The reversal of topological collapse at high temperature (Figure~\ref{fig:thermal_analysis}, Right) illustrates the competition between entropic forces (favoring the high-volume sphere, Proposition~\ref{prop:entropy_scaling}) and algebraic constraints.

\section{Discussion and Outlook}

This work establishes a framework for Dynamical Non-Commutative Algebraic Geometry (DNCAG). We have demonstrated that non-commutativity induces dimensional inflation, characterized the dynamics of these manifolds under central modulation, and analyzed the mechanism of topological collapse under non-central perturbations using gradient flow and thermodynamic formalisms.

The analysis reveals a unified picture across the real normed division algebras:
\begin{enumerate}
    \item The geometry of central roots is determined by the automorphism group of the algebra (Generalized Inflation Theorem).
    \item Central dynamics are governed by the auxiliary polynomial and its discriminant, independent of associativity (due to power-associativity).
    \item Topological collapse is driven by symmetry reduction (Generalized Symmetry Reduction Theorem) and the Alignment Principle (Localization Theorem), which holds due to alternativity.
    \item The dynamics of collapse follow a gradient flow characterized by Critical Slowing Down ($T \propto \epsilon^{-2}$).
    \item The collapse can be rigorously characterized as a thermodynamic phase transition (Entropy Scaling Law and Order Parameter analysis).
\end{enumerate}

Future research will focus on the detailed classification of the $G_2$ collapse hierarchy in $\mathbb{O}$, utilizing the intermediate subalgebras (e.g., $\HH$ within $\OO$), and the extension of this framework to non-division algebras such as Clifford and Matrix algebras.

\backmatter

\end{document}